\documentclass[paper=a4, fontsize=11pt]{scrartcl} % A4 paper and 11pt font size

\usepackage[T1]{fontenc} % Use 8-bit encoding that has 256 glyphs
\usepackage{fourier} % Use the Adobe Utopia font for the document - comment this line to return to the LaTeX default
\usepackage[english]{babel} % English language/hyphenation
\usepackage{amsmath,amsfonts,amsthm} % Math packages
\usepackage{amsmath, calligra, mathrsfs}
\usepackage{amssymb}
\usepackage[mathscr]{euscript}
\usepackage{verbatim}
\usepackage{wasysym}
\usepackage[dvipsnames]{xcolor}
\usepackage{mdframed} % documentation "The mdframed package" on https://www.ps.uni-saarland.de/Publications/documents/Aiken_2009_File.pdf
\usepackage{soulutf8}
\usepackage{stmaryrd}
\usepackage{tikz-cd}
\usepackage{hyperref}
\usepackage{lipsum} % Used for inserting dummy 'Lorem ipsum' text into the template

\usepackage{sectsty} % Allows customizing section commands
\allsectionsfont{\centering \normalfont\scshape} % Make all sections centered, the default font and small caps

\usepackage{fancyhdr} % Custom headers and footers
\usepackage{xurl}
\newcommand*{\sheafhom}{\mathcal{H}\kern -.5pt om}
\numberwithin{equation}{section} % Number equations within sections (i.e. 1.1, 1.2, 2.1, 2.2 instead of 1, 2, 3, 4)
\numberwithin{figure}{section} % Number figures within sections (i.e. 1.1, 1.2, 2.1, 2.2 instead of 1, 2, 3, 4)
\numberwithin{table}{section} % Number tables within sections (i.e. 1.1, 1.2, 2.1, 2.2 instead of 1, 2, 3, 4)

\newtheorem{thm}{Theorem}[section]
\newtheorem{cor}[thm]{Corollary}
\newtheorem{prop}[thm]{Proposition}

\theoremstyle{definition}
\newtheorem{defn}[thm]{Definition}

\newtheorem{exmp}[thm]{Example}

\theoremstyle{remark}
\newtheorem{rem}[thm]{Remark}

\DeclareMathOperator{\Var}{Var}

\DeclareMathOperator{\Sym}{Sym}

\DeclareMathOperator{\sgn}{sgn}
\DeclareMathOperator{\Tr}{Tr}

\DeclareMathOperator{\Conf}{Conf}

\newcommand{\horrule}[1]{\rule{\linewidth}{#1}} % Create horizontal rule command with 1 argument of height

\title{	
	\normalfont \normalsize 
	\textsc{ } \\ [25pt] % Your university, school and/or department name(s)
	\horrule{0.5pt} \\[0.4cm] % Thin top horizontal rule
	\huge Graph coloring-related properties of (generating functions of) Hodge--Deligne polynomials % Generating functions of Hodge--Deligne polynomials and colorings of acyclic directed graphs
	\horrule{2pt} \\[0.5cm] % Thick bottom horizontal rule
}

\author{Soohyun Park} % Your name

\date{\normalsize March 22, 2022} % Today's date or a custom date

\begin{document}
	
	\maketitle

	\begin{abstract}
		\noindent Motivated by a connection between the topology of (generalized) configuration spaces and chromatic polynomials, we show that generating functions of Hodge--Deligne polynomials of quasiprojective varieties and colorings of acyclic directed graphs with the complete graph $K_n$ as the underlying undirected graph. In order to do this, we combine an interpretation of Crew--Spirkl for plethysms involving chromatic symmetric polynomials using colorings of directed acyclic graphs to plethystic exponentials often appearing Hodge--Deligne polynomials of varieties. Applying this to a recent result of Florentino--Nozad--Zamora, we find a connection between $GL_n$-character varieties of finitely presented groups and colorings of these ``complete'' directed graphs by (signed) variables used in the Hodge--Deligne polynomials of the irreducible representations.  \\
		
		\noindent As a consequence of the constructions used, generators of Hodge--Deligne polynomials carry over to generators of the plethystic exponentials as a power series in a way that is compatible with this graph coloring interpretation and gives interesting combinatorial information that simplifies when we consider Hodge--Deligne polynomials of birationally equivalent varieties to be equivalent. Finally, we give an application of methods used to symmetries of Hodge--Deligne polynomials of varieties and their configuration spaces and their relation to chromatic symmetric polynomials.
	\end{abstract}
	
	\section{Introduction}
	
	We study connections between colorings of graphs and invariants of varieties satisfying cut and paste properties. For example, consider the configuration space $\Conf^n X \subset X^n$ be the configuration space of $n$ distinct points on a smooth projective variety $X$. Given a variety $X$, let $HD(X)$ be its Hodge--Deligne polynomial $\sum h^{p, q} u^p v^q$. Then, an induction argument implies that $HD(\Conf^n X) = HD(X)(HD(X) - 1) \cdots (HD(X) - (n - 1))$. Note that $HD(\Conf^n X) = p_{K_n}(HD(X))$, where $p_G(\lambda)$ is the chromatic polynomial of a graph $G$ with $\lambda$ available colors. Our objective is to study a connection between colorings of \emph{directed} graphs and Hodge--Deligne polynomials of varieties. \\
	
	More specifically, we consider \emph{generating functions} of Serre $E$-polynomials (which give Hodge--Deligne polynomials in the smooth projective case). The main tool which we use is the plethysm operation between symmetric functions, which were originally involved with the study of $GL_n$-representations (p. 446 -- 447 of \cite{Sta2}). There is an extensive literature on this combinatorial operation with an overview of combinatorial/computational perspectives in a survey of Loehr--Remmel \cite{LR}. A brief overview of the specific definitions and results used in the proof of the main theorem is given in Section \ref{symbackground}. \\
	
	The plethysm operation's connection with generating functions of Serre $E$-polynomials comes from plethystic exponentials of appropriate functions (Definition 2.2 on p. 25 of \cite{F}, Definition \ref{plethexpdef}). For example, this operation has appeared in work of Getzler \cite{G} (Theorem on p. 2 and Corollary 5.7 on p. 16 of \cite{G}) related to equivariant Serre $E$-polynomials of ordered configuration spaces of points of quasiprojective varieties and more recent work of Florentino--Nozad--Zamora \cite{FNZ} on Serre $E$-polynomials of $GL_n$-character varieties of finitely presented groups (Theorem 1.1 on p. 2 of \cite{FNZ}). \\
	
	Using the plethysm operation, we can express generating functions of Hodge--Deligne polynomials as plethysms of symmetric functions. Results of Cho--van Willigenburg \cite{CW} imply that symmetric functions are polynomially generated by generating symmetric functions by chromatic symmetric functions (Theorem 5 on p. 4 of \cite{CW}). Then, we can apply recent results of Crew--Sprikl \cite{CS} on plethysms of chromatic symmetric polynomials (Theorem 2 on p. 7 of \cite{CS}) to show that generating functions of Hodge--Deligne polynomials are determined by colorings of acylic directed graphs by signed variables (Part 1 of Theorem \ref{hdcolorpleth}). A more precise statement is listed below.

	\begin{thm} \label{hdcolorpleth} ~\\
		\vspace{-5mm}
		\begin{enumerate}
			\item Generating functions of Hodge--Deligne polynomials of projective varieties $X$ (and Serre $E$-polynomials in general) are determined by proper colorings of acyclic directed graphs. The underlying undirected graphs can be given by the complete graphs $K_n$ and colorings are given by signed ``variables'' $u^p v^q$ for $p$ and $q$ such that $h^{p, q} \ne 0$. These variables are counted with multiplicity given by the Hodge numbers $h^{p, q}$ of $X$. The generating function is an alternating sum of products of ``colors'' given by the variables. This sum depends on possible acyclic orientations of the $K_n$ and possible signed proper colorings. \\
			
			For example, the number $n = 1 + \ldots + 1$ induces a coloring of the graphs by $n$ copies of $1$ (p. 8 of \cite{CS}). This explained in more detail in the proof of this result in Section \ref{plethgen}. There are also some additional examples of colorings given in Example \ref{gencolorcomp}.
			
			\item If $\Gamma$ is a finitely presented group, the Hodge--Deligne polynomial of the $GL_n$-character variety $X_\Gamma (GL_n)$ of $\Gamma$ is determined by proper colorings of acyclic directed graphs with the undirected graph $K_n$ as an underlying set by signed variables $u^p v^q$ with $h^{p, q}(X_\Gamma^{\text{irr}} (GL_n)) \ne 0$, where $X_\Gamma^{\text{irr}} (GL_n)$ denotes the irreducible representations. The coloring uses the same method as Part 1.
			
			\item Let $F(X, n) \subset \Sym^n X$ be the configuration space of \emph{unordered} distinct points of $X$ and $e(F(X, n), \varepsilon)$ be the equivariant Hodge--Deligne polynomial with respect to the sign representation (p. 3 and p. 8 of \cite{G}).
				\begin{enumerate}
					\item Single variable generating functions of equivariant Hodge--Deligne polynomials of the ordered configuration space with respect to the sign representation are given by the coloring in Part 1 with the signs of the variables reversed. 
					
					\item Replacing the equivariant Hodge--Deligne polynomials in part a by those of unordered configuration spaces, the generating function of the Hodge--Deligne polynomial is a quotient of generating functions similar to those in Part 1. The numerator is given by that polynomial in Part 1 with $1$ replacing $h^{p, q}$ and the denominator given by replacing the ``auxiliary variable'' $y$ by $y^2$ in Definition \ref{plethexpdef}.
				\end{enumerate}

		\end{enumerate}
	\end{thm}
	
	\begin{rem}
		As stated in Corollary \ref{chromext}, the complete graphs in Part 1 of Theorem \ref{hdcolorpleth} can be replaced by a suitable sequence of connected graphs. \\
	\end{rem}
	
	Note that Part 2 of Theorem \ref{hdcolorpleth} also gives a new interpretation of a relation between Serre $E$-polynomials of $GL_n$-character varieties of finitely presented groups and the associated irreducible representations and has an extension to certain equivariant Serre $E$-polynomials of configuration spaces of points of quasiprojective varieties (Part 2 and Part 3 of Theorem \ref{hdcolorpleth}). An extension involving colorings of a wider range of graphs is given in Corollary \ref{chromext}. \\
	
	Following the proof of Theorem \ref{hdcolorpleth} in Section \ref{plethgen}, we indicate how generators of Hodge--Deligne polynomials carry over to those of plethystic exponentials in a way compatible with the coloring interpretation (Theorem \ref{uniquegeom}) in Section \ref{uniquegeomsect}. 
	
	\begin{thm} \label{uniquegeom} ~\\
		\vspace{-5mm} 
		\begin{enumerate}
			\item Given a variety $X$ of dimension $n$, let $HD(X)(u, v) y^n$ be its graded Hodge--Deligne polynomial. The plethystic exponential (Definition \ref{plethexpdef}) of the graded Hodge--Deligne polynomial of a smooth projective variety $X$ is uniquely defined by colorings of acylic orientations of $K_n$ (and suitable collections of connected graphs by Corollary \ref{chromext}) by $2, 3,$ or $4$ ``variables'' used to build the Hodge--Deligne polynomial of $X$. The colors are variables (with multiplicity) used to build Hodge--Deligne polynomials of $\mathbb{CP}^1$, $\mathbb{CP}^2$, and an elliptic curve respectively. \\
			
			Note that this polynomial generation holds even if the original plethystic exponential definition only implies multiplicative generation by Hodge--Deligne polynomials of monomials in these in these terms induced by addition. The further structure induced by multiplication of the original Hodge--Deligne polynomials follows from properties of symmetric functions. 
			
			\item Applying Corollary \ref{chromext}, the statement in Part 1 still holds after replacing the complete graphs $\{ K_n \}$ by any collection of connected graphs $\{ G_k \}$ with $G_k$ containing $k$ vertices. This comes from writing each complete symmetric polynomial $h_n$ as a polynomial in the $X_{G_k}$. 
		\end{enumerate}
	\end{thm}
	
	\begin{rem}
		Regarding realizability by actual compact K\"ahler manifolds, anything satisfying the Hodge symmetry relation $h^{p, q} = h^{q, p}$ and Serre duality relation $h^{p, q} = h^{n - p, n - q}$ is taken $\pmod m$ is realized by such a variety  (Theorem 2 on p. 2428 of \cite{PaS}) and assuming that the middle Hodge number $h^{m, m}$ is large relative to $m$ is sufficient for the integers themselves (Theorem 1 on p. 296 of \cite{Sch}). \\
	\end{rem}
	
	Other examples of colorings are on p. 8 of \cite{CS}. As mentioned in the statement of Part 1 of Theorem \ref{hdcolorpleth}, the ideas are explained in further detail in the proof of Part 1 Theorem \ref{hdcolorpleth}. This includes explicit combinatorial information on the generation in the plethystic exponential and simplifications modulo birational equivalence (Proposition \ref{plethpolycoeff}). 
	
	 \begin{prop} ~\\
	 	\vspace{-3mm}
		\begin{enumerate}
			\item The degree $n$ term (with respect to the auxiliary variable $t$) of the plethystic exponential of a Hodge--Deligne polynomial of a smooth projective variety can be decomposed as products of complete homogeneous polynomials with terms of the form $u^p v^q, u^q v^p, u^{n - p} v^{n - q}$ substituted in for the variables (only taking the first two in the case $p + q = n$). The degrees of the terms depend on partitions of $n$ into $h^{p, q}$ parts. \\
			
			\item The coloring polynomials (in the context of Theorem \ref{hdcolorpleth}) of Hodge--Deligne polynomials of the varieties $\mathbb{CP}^1, \mathbb{CP}^2$, and any elliptic curve $X$ polynomially generate the plethystic exponential of the (graded) Hodge--Deligne polynomial of any compact K\"ahler variety via nested determinants with ($\mathbb{Q}$-linear combinations of) coloring polynomials of complete graphs (or the graphs $G_k$ determining the chosen chromatic symmetric function basis). The determinants which we substitute into are determined by heights of rim hooks (p. 61 of \cite{ER}). \\
			
			\item Suppose we consider graded Hodge--Deligne polynomials of birationally equivalent varieties to be equivalent to each other. Let $h_r$ be the complete homogeneous polynomial of degree $r$ and $p_a$ be the $8*{\text{th}}$ single. If $n \ge 2$, the degree $n$ part (in the auxiliary variable) of a plethystic exponential is a $\mathbb{Z}$-linear combination of products of terms of the form $h_r \odot p_a$ with the number of terms in the product depending on a partition of $n$ into parts with the number of parts depending on the Hodge numbers as in Part 1. These generators depend the difference of the Hodge--Deligne polynomials of any elliptic curve and $\mathbb{CP}^1$ and the Hodge--Delinge polynomial of $\mathbb{CP}^1$ itself. \\
		\end{enumerate}
	\end{prop}

	Finally, we will give a connection to symmetries of Hodge--Deligne polynomials themselves in Section \ref{chromsd} (Proposition \ref{chromhdsym}). This will involve expressing the Serre duality symmetry of Hodge numbers $h^{p, q} = h^{n - p, n - q}$ and equivariant Hodge--Deligne polynomials of ordered configuration spaces in terms of chromatic symmetric polynomials.

	\section*{Acknowledgements}
	I am very thankful to my advisor Benson Farb for helpful discussions and encouragement throughout the project. Also, I would like to thank him for extensive comments on preliminary drafts of this paper.
	
	\section{Combinatorial and geometric definitions} \label{symbackground}
	
	Before giving the proof of the main result (Theorem \ref{hdcolorpleth}), we give a short outline of definitions which will be used to study the connection of plethysm with generating functions of $E$-Serre polynomials of (quasi)projective varieties. As mentioned, there are many references exploring plethysms in further depth such as the recent survey of Loehr--Remmel \cite{LR} from a combinatorial/computational perspective. \\

	We first define the plethysm operation. A discussion with some motivation involving connections with $GL_n$-representations is in p. 446 -- 447 of \cite{Sta2} and a computation/combinatorial survey is given in \cite{LR}.
	
	\begin{defn} (p. 2 -- 3 in Section 2.2 of \cite{CS}) \label{plethdef} \\
		Let $\Lambda \subset \mathbb{C}[[x_1, x_2, \ldots]]$ be the \textbf{algebra of symmetric functions}. This is the subalgebra of functions $f$ of bounded degree which are symmetric. This algebra has a natrual structure as a graded algebra $\Lambda = \bigoplus_{d = 0}^\infty \Lambda^d$ with $\Lambda^d$ as the the homogeneous symmetric functions of degree $d$. \\
		
		The \textbf{plethysm} $f \odot g$ of two symmetric functions $f, g \in \Lambda$ is defined using the following operations involving constant functions and power sums. This can be used to define it on the entire ring of symmetric functions since this ring is generated by power sums with $p_k$ denoting the $k^{\text{th}}$ power sum (Corollary 7.7.2 on p. 298 of \cite{Sta2}). 
		
		\begin{itemize}
			\item $c \odot f = c$
			
			\item $p_n \odot c = c$
			
			\item $p_n \odot p_m = p_{nm}$
			
			\item $p_n \odot (f + g) = (p_n \odot f) + (p_n \odot g)$
			
			\item $p_n \odot (fg) = (p_n \odot f) \cdot (p_n \odot g)$
			
			\item $(fg) \odot h = (f \odot h) \cdot (g \odot h)$
			
			\item $(f + g) \odot h = (f \odot h) + (g \odot h)$
		\end{itemize}
		
	\end{defn}
	
	The objects serving as a connection between the combinatorial objects induced by plethysm and the Serre/Hodge--Deligne polynomials are chromatic symmetric polynomials, which are a generalization of chromatic polynomials specializing to them (Proposition 2.2 on p. 169 of \cite{Sta1}).
	
	\begin{defn} (Definition 1 on p. 3 of \cite{CW}, Definition 2.1 on p. 168 of \cite{Sta1}) \\
		For a graph $G$ with vertex set $V(G) = \{ v_1, \ldots, v_n \}$ and edge set $E(G)$, the \textbf{chromatic symmetric function of $G$} is defined to be \[ X_G = \sum_\kappa x_{\kappa(v_1)} \cdots x_{\kappa(v_n)}, \] where the sum is over all proper colorings $\kappa$ of $G$.
	\end{defn}
	
	On the geometric side, we recall the definition of (equivariant) Serre $E$-polynomials/Hodge--Deligne polynomials.
	
	\begin{defn}(p. 1 of \cite{G}) \label{equivser}
		\begin{enumerate}
			\item Given a quasiprojective variety $X$ over $\mathbb{C}$, let the \textbf{Serre $E$-polynomial/Hodge--Deligne polynomial} $e(X)$ be the polynomial satisfying the following properties:
			
			\begin{enumerate}
				\item If $X$ is projective and smooth, $e(X)$ is the (signed) Hodge polynomial \[ e(X) (u, v)=  \sum_{p, q = 0}^\infty (-u)^p (-v)^q \dim H^{p, q}(X, \mathbb{C}).  \]
				
				\item If $Z \subset X$ is a closed subvariety, then $e(X) = e(X \setminus Z) + e(Z)$. 
			\end{enumerate}
			
			\item If a finite group $G \circlearrowright X$, define the \textbf{equivariant Serre $E$-polynomial/Hodge--Deligne polynomial} by the formula \[ e_g(X) = \sum_{p, q = 0}^\infty \sum_i (-1)^i \Tr(g | H_c^i(X, \mathbb{C})^{p, q}). \]
			
		\end{enumerate}
		
	\end{defn}
	
	The connection to Serre $E$-polynomials to plethysm-related constructions is given through plethystic exponentials, which are defined below.

	\begin{defn}(Definition 2.2 on p. 25 of \cite{F}) \label{plethexpdef} \\
		Let $R = \mathbb{C}[[x_1, x_2, \ldots, x_n, y]]$ be a ring of formal power series in $n + 1$ variables $x_1, \ldots, x_n$ and let $y$ and $R^0 \subset R$ be the ideal of formal power series with constant term equal to $0$. Given $f(x, y) \in R^0$, the \textbf{plethystic exponential} of $f$ is \[ PE[f] := \exp \circ \Psi[f], \] where $\Psi[f](x, y) := \sum_{m \ge 1} \frac{f(x^m, y^m)}{m}$. 
	\end{defn}

	We will relate this to the plethysm operation in the beginning of the proof of Theorem \ref{hdcolorpleth}. \\
	
	\begin{rem}
		The single variable version of plethystic exponentials is used in the exponential formulation of the Hasse--Weil zeta function with $\mathbb{F}_{q^m}$-point counts substituted in (p. 26 of \cite{CNS}). These power series themselves can also be used to determine Hodge numbers (part b of Theorem 2.2.5 on p. 27 -- 28 of \cite{CNS}).
	\end{rem}

	\section{Generating functions of Hodge--Deligne polynomials and colorings of acyclic directed graphs} \label{plethgen}
	
	In this section, we prove our main result connecting generating functions of $E$-Serre polynomials with colorings of acyclic directed graphs (Theorem \ref{hdcolorpleth}). The idea is to connect the generating functions of these polynomials (which give plethystic exponentials) with \emph{combinatorial} properties of plethysm operations involving chromatic symmetric polynomials.

	\begin{proof} (Proof of Theorem \ref{hdcolorpleth}) \\
		\begin{enumerate}
			\item
			We use some standard properties of symmetric functions to write plethystic exponentials involved in generating functions of Hodge--Deligne and Poincar\'e polynomials (e.g. in \cite{F}) as generating functions of plethysms of actual functions. Let $f$ be a formal power series in $k[[x_1, \ldots, x_n, y]]$ with constant term equal to $0$. In our case, we will eventually take $f$ to be a Hodge--Deligne polynomial (which ends after finitely many terms). Using the notation in Definition \ref{plethdef}, the plethystic exponential of $f$ is $PE[f] = \exp \circ \Psi[f]$, where $\Phi[f](x, y) = \sum_{m \ge 1} \frac{f(x^m, y^m)}{m}$ (itself a power series with constant term equal to $0$). This actually has an interpretation involving the complete homogeneous symmetric polynomials $h_n = \sum_{1 \le i_1 \le \cdots \le i_k \le n} x_{i_1} \cdots x_{i_k}$ and power sums $p_r = x_1^r + \cdots + x_n^r$. In identity connecting $h_n$ and $p_r$ (description of $H(t)$ on p. 96 of \cite{Mac}) implies that \[ \sum_{n \ge 0} h_n t^n = \exp \left( \sum_{i \ge 1} \frac{p_i}{i} t^i \right). \] Given a power series $g$ with coefficients in $\mathbb{Q}$, this implies that \[ \sum_{n \ge 0} (h_n \odot g) t^n = \exp \left( \sum_{i \ge 1} \frac{p_i \odot g}{i} t^i \right), \]  where $a \odot b$ denotes the plethysm operation between $a$ and $b$. \\
			
			Since $(p_k \odot f)(x_1, x_2, \ldots) = f(x_1^k, x_2^k, \ldots)$, the information above can be used to (formally) rewrite the plethystic exponential as 
			
			\begin{align*}
				PE[f] &= \exp \left( \sum_{m \ge 1} \frac{f(x^m, y^m)}{m} \right) \\ 
				&= \exp \left( \sum_{m \ge 1} \frac{p_m \odot f}{m} \right) \\
				&= \sum_{n \ge 0} (h_n \odot f) \\
				&= \sum_{n \ge 0}  \frac{(-1)^n}{n!} (X_{K_n} \odot (-f)), 
			\end{align*}

			where the last equality is from the proof of Corollary 4 on p. 12 of \cite{CS} and $X_G$ is the chromatic symmetric polynomial associated to $G$.  \\
			
			Since the plethysm in each term involves a chromatic symmetric polynomial, we can apply the following recent result which gives an interpretation in terms of proper colorings of directed acyclic graphs. It is a decomposition of the plethysm operation between a chromatic symmetric polynomial for a weighted graph and a power series in terms of proper colorings of acyclic graphs.
			
			\begin{thm}(Crew--Spirkl, Theorem 2 on p. 7 of \cite{CS}) \label{chromplethint} \\
				For $f$ an expression for which $\Var(f)$ is defined, prescribe a total ordering $<$ on the elements of $\Var(f)$. Then \[ X_{(G, w)} \odot f = \sum_{(\gamma, \kappa)} \prod_{v \in V(G)} \sgn(\kappa(v)) \kappa(v)^{w(v)}, \] where the sum runs over all ordered pairs consisting of an acyclic orientation $\gamma$ of $G$ and a map $\kappa : V(G) \longrightarrow \Var(f)$ such that
				
				\begin{itemize}
					\item If $u \xrightarrow{\gamma} v$, then $\kappa(u) \le \kappa(v)$. 
					
					\item Whenever $\kappa(u) = \kappa(v)$ and $\sgn(\kappa(u)) = 1$, $uv \notin E(G)$.
				\end{itemize}
			\end{thm}
			
			The polynomial $X_{(G, w)}$ is a generalization of the usual chromatic symmetric polynomial for weighted graphs and setting $w(v) = 1$ for all $v \in v(G)$ gives the unweighted chromatic symmetric polynomial (p. 5 of \cite{CS}). Roughly speaking, a suitable expression $f$ can be taken to be some polynomial or power series and the ``colorings'' here are labels by signed variables (with multiplicity) used to write down the input polynomial $f$. Note that $\sgn(v)$ for a variable $v$ is simply a sign (equal to $\pm 1$) assigned to each variable used to write down $f$ which is reversed for $-f$. More precise definitions for $\Var(f)$ are given in Definition 1 on p. 7 of \cite{CS}. \\
			
			Applying Theorem \ref{chromplethint} to the \emph{unweighted} graph $K_n$ above while setting $w(v) = 1$ for all $v \in V(K_n)$ (p. 5 -- 6 of \cite{CS}), this implies that \[ PE[f] =  \sum_{n \ge 0} \sum_{\gamma_n, \kappa_{-f}} \prod_{v \in V(K_n)} \sgn(\kappa_{-f}(v)) \kappa_{-f}(v), \]
			
			where $\gamma_n$ denotes possible acyclic orientations of the graph $K_n$ and $\kappa_{-f}$ denotes possible proper colorings (with respect to $\gamma_n$) by signed variables from $-f$. \\

			In particular, each term of this final sum is an alternating sum of products of vertices involved in proper colorings of acylic directed graphs with $K_n$ as an underlying undirected graph. The dependence on the starting function $f$ is in the choice of available proper colorings $\kappa_{-f}$ with signed variables of $-f$ (described more precisely in Definition 1 and Theorem 2 on p. 7 of \cite{CS}) while the acyclic orientations $\gamma$ are entirely determined by $K_n$ (i.e. the choice of $G$ in the chromatic symmetric polynomial $X_G$). \\ 
			
			To obtain the statement for Hodge--Deligne polynomials we use the fact that \[ PE \left(  \left( \sum_{p, q \ge 0} a_{p, q} u^p v^q \right) y \right) = \prod_{p, q \ge 0} (1 - u^p v^q y)^{-a_{p, q}}. \] One source for this is Lemma 4.7 on p. 11 of \cite{FNZ}.
			
			\item This statement for Hodge--Deligne polynomials of character varieties follows from applying the proof of Part 1 to the main result of \cite{FNZ}:
			
				\begin{thm}(Florentino--Nozad--Zamora, Theorem 1.1 on p. 2 of \cite{FNZ}) \\
					Let $\Gamma$ be a finitely generated group, $X_\Gamma (GL_n)$ be the $GL_n$-character variety of $\Gamma$, and $X_\Gamma^{\text{irr}} (GL_n)$ be the irreducible representations. Then, in $\mathbb{Q}[u, v][[t]]$, \[ \sum_{n \ge 0} E(X_\Gamma GL_n; u, v) t^n = PE \left( \sum_{n \ge 1} E(X_\Gamma^{\text{irr}} GL_n; u, v) t^n \right). \]
				\end{thm}
			
			\item This part follows from applying the proof of Part 1 to the following result of Getzler \cite{G}:
			
				\begin{thm} (Getzler, Corollary 5.7 on p. 16 of \cite{G}) \\
					Let $F(X, n) \subset \Sym^n X$ be the configuration space of \emph{unordered} distinct points of $X$ and $e(F(X, n), \varepsilon)$ be the equivariant Hodge--Deligne polynomial with respect to the sign representation (p. 3 and p. 8 of \cite{G}). If $e(X) = \sum_{p, q} h^{p, q} u^p v^q$ is the Serre $E$-polynomial of $X$, then \[ \sum_{n = 0}^\infty t^n e(F(X, n), \varepsilon) = \prod_{p, q}^\infty (1 + t u^p v^q)^{h_{p, q}} \] and \[ \sum_{n = 0}^\infty t^n e(F(X, n)/S_n) = \prod_{p, q \ge 0} \left( \frac{1 - t^2 u^p v^q}{1 -t u^p v^q} \right). \] 
				\end{thm}
			
			 We apply this result while substituting in $-t$ in place of $t$ in the first line. The second part is a quotient of results in Part 1 with the numerator corresponding to replacing $t$ with $t^2$.
			
		\end{enumerate}

	\end{proof}
	
	The analysis in Theorem \ref{hdcolorpleth} can be repeated using chromatic symmetric polynomials other than $X_{K_n}$. This is based on the fact that there is a finite sequence chromatic symmetric polynomials $X_G$ form a $\mathbb{Q}$-basis of $\Lambda^d$ for each $d$ (Definition \ref{plethdef}).

	\begin{cor} \label{chromext}
		Statements analogous to those in Theorem \ref{hdcolorpleth} hold when we replace the complete graphs $\{K_n\}$ by suitable sequence of connected graphs $\{ G_k \}$ with each $G_k$ containing $k$ vertices. 
	\end{cor}
	
	\begin{proof}
		In the beginning of the proof of Part 1 of Theorem \ref{hdcolorpleth}, we could have written $h_n \odot g$ using a different $\mathbb{Q}$-basis of chromatic symmetric polynomials using the following result. 
			
			\begin{thm}(Cho--van Willigenburg, Theorem 5 on p. 4 of \cite{CW}) \label{chromgen} \\
				Let $\{ G_k \}$ be a set of connected graphs such that $G_k$ has $k$ vertices for each $k$. Then, $\{ X_{G_\lambda} : \lambda \dashv n \}$ is a $\mathbb{Q}$-basis of $\Lambda^n$. We also have that $\Lambda = \mathbb{Q}[X_{G_1}, X_{G_2}, \ldots]$ and the $X_{G_k}$ are algebraically independent over $\mathbb{Q}$. 
			\end{thm}
			
		We can make this more concrete using products of power sums. Note that a version of Newton's identities relating homogenenous complete symmetric polynomials with power sums states that $n h_n = p_n + h_1 p_{n - 1} + \ldots + h_{n - 1} p_1$ (p. 4 of \cite{G}). To reduce to this case, it suffices to write $h_n$ as a linear combination of chromatic symmetric polynomials based on writing $h_n$ as a $\mathbb{Q}$-linear combination of products of power sums $p_\lambda = p_{\lambda_1} \cdots p_{\lambda_r}$ with $\lambda_1 + \ldots + \lambda_r = n$. Using the version of Newton's identity for homogeneous complete symmetric polynomials, we can write $h_n = \sum_{h \dashv n} z_\lambda^{-1} p_\lambda$, where $z_\lambda = 1^{m_1} m_1! \cdot 2^{m_2} m_2! \cdot \cdots$ and $z_\lambda^{-1}$ is the number of permutations of $n$ objects of fixed cycle type $\lambda = (1^{m_1} 2^{m_2} \cdots )$ divided by $n!$ (p. 298 -- 299 and Proposition 7.7.6 on p. 301 of \cite{Sta2}). The expression in terms of power sums is the main tool used to express the algebra of symmetric polynomials in terms of chromatic symmetric polynomials in \cite{CW}. \\

		As in the setup of the proof of this Theorem \ref{chromgen} on p. 4 of \cite{CW}, let $\{ G_k \}$ be a set of connected graphs such that $G_k$ has $k$ vertices for each $k$. Writing $X_{G_\lambda} = \sum_{\mu \le \lambda} c_{\lambda \mu} p_\lambda$, we can use a change of variables to write each term of the sum $PE[f]$ as a $\mathbb{Q}$-linear combination of plethysms of the form $X_G \odot f$, which have a decomposition analogous to the one listed in Theorem \ref{hdcolorpleth}. The main difference would be that the weightings in the weighted graph chromatic symmetric polynomial may be nontrivial. 
		
	\end{proof}
	
	\section{Uniqueness and connections with the construction problem for Hodge numbers} \label{uniquegeomsect}
	
	Given the computations above, it is natural to ask whether they are reversible in some sense. Since the plethystic exponential actually induces a \emph{group isomorphism} from the (additive group of) of formal power series with 0 as constant term to (multipliticative group of) formal power series with 1 as constant term (Remark 2.2 on p. 26 of \cite{F}), this is indeed the case if the Hodge numbers are equal. \\
	
	This means that the answer essentially depends on the construction problem for Hodge numbers. While multiplication of Hodge--Deligne polynomials doesn't literally yield a homomorphism on the plethystic exponential side, properties of symmetric functions still allow us to show that generators of Hodge--Deligne polynomials as an algebra can carry over to power series polynomially generating the plethystic exponential side. In particular, this reduces the number of colors required to be at most 4. These relations are specified in the statement of Theorem \ref{uniquegeom}, which is proved below.

	\begin{proof} (Proof of Theorem \ref{uniquegeom}) \\
		The proof will focus on Part 1 since Part 2 follows from Part 1 after applying Corollary \ref{chromext}. \\
		Given a variety $X$, let $HD(X)$ be its Hodge--Deligne polynomial. Moving to \emph{graded} formal Hodge--Deligne polynomials $HD(X)z^n$ for $X$ ($\dim X = n$) compatible with the Hodge symmetry relation $h^{p, q} = h^{q, p}$ and Serre duality relation $h^{p, q} = h^{n - p, n - q}$, we can apply construction-related results of Kotschick--Schreieder \cite{KS} and Paulsen--Schreieder \cite{PaS}. The setup in the proof of Theorem \ref{hdcolorpleth} in Section \ref{plethgen} still applies since \[ \sum_{n \ge 0} h_n t^n = \exp \left( \sum_{i \ge 1} \frac{p_i}{i} t^i \right), \] which implies that  \[ \sum_{n \ge 0} (h_n \odot g) t^n = \exp \left( \sum_{i \ge 1} \frac{p_i \odot g}{i} t^i \right), \]  where $a \odot b$ denotes the plethysm operation between $a$ and $b$. \\
		
		Note that the former implies the latter since \[ e^x = 1 + x + \frac{x^2}{2!} + \ldots \] and $(fg) \odot h = (f \odot h) (g \odot h)$. Since $(f + g) \odot h = f \odot h + g \odot h$ and $c \odot f = c$ for any constant $c$, we have that
		
		\begin{align*}
			e^x \odot g &= 1 \odot g + x \odot g + \frac{x^2}{2!} \odot g + \ldots \\
			&= 1 + (x \odot g) + \frac{(x \odot g)^2}{2!} + \ldots \\
			&= e^{x \odot g}.
		\end{align*}

		Combining Corollary 3 on p. 646 of \cite{KS} and Theorem \ref{hdcolorpleth}, the plethystic exponentials coming from Hodge--Deligne polynomials are generated by products of colorings of $K_n$ with an acyclic orientation by variables used to build $HD(X) = 1 + u + v + uv$ for an elliptic curve $X$, $HD(\mathbb{CP}^1) = 1 + uv$, and $HD(\mathbb{CP}^2) = 1 + uv + u^2 v^2$ and there are no nontirival multiplicative relations between these products. The situation is simplified when we quotient out by differences of birational smooth complex projective varieties since we only need the generators $X$ and $\mathbb{CP}^1$ (Theorem 7 on p. 647 and Corollary 3 on p. 646 of \cite{KS}). \\

		In order to incorporate the multiplicative structure of these rings purely using the original generators themselves without taking products, we would need to make use of formulas involving Schur functions for the plethysm operation $h_n \odot (fg)$ (Part 3 on p. 137 of \cite{Mac}) between a complete symmetric polynomial of degree $n$ and a product of symmetric functions. While the coloring variables reduce to the building blocks $\mathbb{CP}^1$, $\mathbb{CP}^2$, and an elliptic curve $X$, the graphs involved would vary depending on the Schur functions and their expression as a linear combination of chromatic symmetric polynomials (Theorem \ref{chromgen}) instead of simply looking at complete graphs $K_n$. \\
		
		Combining the identity \[ \sum_{n \ge 0} (h_n \odot g) t^n = \exp \left( \sum_{i \ge 1} \frac{p_i \odot g}{i} t^i \right) \] with the identities \[ h_n \odot (fg) = \sum_{|\lambda| = n} (s_\lambda \odot f) (s_\lambda \odot g) \] and \[ h_n \odot (f + g) = \sum_{k = 0}^n (h_k \odot f) (h_{n - k} \odot g) \] (equation (8.8) on p. 136, p. 177 of \cite{LR}, Part 3 on p. 137 of \cite{Mac}), we find that the plethysm operations $h_n \odot f$ with $f = HD(X)= (1 + u + v + uv)$ for an elliptic curve $X$, $f = HD(\mathbb{CP}^1) = (1 + uv) $, and $f = HD(\mathbb{CP}^2) = 1 + uv + u^2 v^2$ still generate everything naturally using multiplication and addition (i.e. as generators of an algebra) after using the plethystic exponential. Note that this still works for multiplication since we can write $s_\lambda$ as a polynomial in $h_r$ by the Jacobson--Trudi formula $s_\lambda = \sum_{\sigma \in S_N} \varepsilon(\sigma) h_{\lambda + \omega(\sigma)}$ with $\omega(\sigma)_j = \omega(j) - j$ and $\varepsilon(\sigma)$ equal to the sign of the permutation $\sigma$ (Theorem 7.16.1 on p. 342 of \cite{Sta2}) and we can use the identities $(f + g) \odot h = f \odot h + g \odot h$ and $(fg) \odot h = (f \odot h)(g \odot h)$.\\
		
		If we chose to use a different basis of generators from $X_{K_n}$ using chromatic symmetric polynomials $X_{G_k}$ from a collection of connected graphs $\{ G_k \}$ with each $G_k$ containing $k$ vertices, we could have used the decomposition of $s_\lambda$ as a polynomial in the $X_{G_k}$ (Theorem \ref{chromgen}) and proceeded as previously to show that everything is still generated as a polynomial by colorings induced by $HD(\mathbb{CP}^1)$, $HD(\mathbb{CP}^2)$, and $HD(X)$ for an elliptic curve $X$. \\
		
	\end{proof}

	 In the particular case of the generators $X, \mathbb{CP}^1$, and $\mathbb{CP}^2$, we can give a more concrete explanation of why the chromatic symmetric polynomial decomposition in Theorem \ref{hdcolorpleth} is consistent with the simplification  \[ PE \left(  \left( \sum_{p, q \ge 0} a_{p, q} u^p v^q \right) y \right) = \prod_{p, q \ge 0} (1 - u^p v^q y)^{-a_{p, q}} \] used in sources such as Lemma 4.7 on p. 11 of \cite{FNZ}. Note that there is some simplification coming from the fact that the coefficient of each nonzero term is equal to $1$ (no repeats in variables counted with multiplicity and all the same sign). 
	 
	 \begin{exmp} \label{gencolorcomp} ~\\
	 	\vspace{-3mm} 
	 	
	 	For the generators given by $\mathbb{CP}^1$, $\mathbb{CP}^2$, and an elliptic curve $X$, we can check that the decomposition given in Theorem \ref{hdcolorpleth} is consistent with the identity \[ PE \left(  \left( \sum_{p, q \ge 0} a_{p, q} u^p v^q \right) y \right) = \prod_{p, q \ge 0} (1 - u^p v^q y)^{-a_{p, q}}. \] 
	 	
	 	\begin{enumerate}
	 		\item The identity above implies that $PE[HD(\mathbb{CP}^1) y] = PE[(1 + uv) y] = (1 - y)^{-1} (1 - uv y)^{-1} = (1 + y + y^2 + \ldots) (1 + uvy + (uvy)^2 + \ldots) $. The terms of degree $k$ are given by $\sum_{i + j = k} y^i (uv y)^j$. This means that the coefficient of $y^k$ is $\sum_{0 \le j \le k} (uv)^j$. Let $g(u, v) = HD(\mathbb{CP}^1) = 1 + uv$.  Fix an ordering $1 < uv$ and suppose that $1$ and $uv$ both have a positive sign. Using the proof of Theorem \ref{hdcolorpleth}, we find that \[ PE[HD(\mathbb{CP}^1) y]  = \sum_{n \ge 0} \frac{(-1)^n}{n!}  \left( \sum_{\gamma_n, \kappa_{-g}} \prod_{v \in V(K_n)} \sgn(\kappa_{-g}(v)) \kappa_{-g}(v) \right) y^n, \] where $\gamma_n$ corresponds to acyclic orientations of the complete graph $K_n$ and $\kappa_{-g}$ denotes colorings of the vertices of $K_n$ such that $\kappa_{-g}(u) \le \kappa_{-g}(v)$ is $u \xrightarrow{\gamma_n} v$. Note that the sign condition mentioned in Theorem \ref{chromplethint} does not apply since every single variable in $-g$ has a negative sign. \\
	 		
	 		Since an acyclic orientation of $K_n$ amounts to a total ordering on the vertices of $K_n$, there are $n!$ acyclic orientations. Also, note that the $\sgn(\kappa_{-g}(v))$ term induces multiplication by $(-1)^n$ since every single vertex is labeled by a variable with a negative sign. Given a particular coloring compatible with an acylic orientation, the product inside is equal to $(uv)^r$, where $r$ is the total number of vertices labeled by $-uv$. Each $r$ only occurs once since the ordering condition implies that $uv$ labels come in a string starting from the $n^{\text{th}}$ (largest) vertex and moving backwards until we stop at some point and label the rest of the vertices by $-1$. This induces a sum $1 + uv + (uv)^2 + \ldots + (uv)^r$, which is consistent with the computation above. \\
	 		
	 		For the other two cases with $\mathbb{CP}^2$ and  the elliptic curve $X$ will use a more combinatorial proof. \\
	 		
	 		\item In the case of $HD(\mathbb{CP}^2) = 1 + uv + u^2 v^2$, we can repeat the above with the same idea of partitioning $n$ vertices into blocks according to whether they are labeled by $-1$, $-uv$, or $-u^2 v^2$. Next, we expand $(1 - a)^{-1}  (1 - uv a)^{-1} (1 - u^2 v^2 a)^{-1} = (1 + a + a^2 + \ldots) (1 + (uv a) + (uva)^2 + \ldots) (1 + (u^2 v^2 a) + (u^2 v^2 a)^2 + \ldots)$ and substituting in $a = y^2$. Collecting terms of degree $2r$ in $y$, we find the term indicated in the acyclic coloring interpretation of Theorem \ref{hdcolorpleth}.  \\
	 		
	 		\item Given an elliptic curve $X$, the same combinatorial idea of partitions and collecting terms in power series expansions as in Part 2 applies for $HD(X) = 1 + u + v + uv$. \\
	 	\end{enumerate}
	 \end{exmp}
	 
	 Here is some more explicit combinatorial information on how the dependencies on the generators $\mathbb{CP}^1$, $\mathbb{CP}^2$, and the elliptic curve $X$ carry over to the plethystic exponential. \\
	 
	 \begin{prop} \label{plethpolycoeff} ~\\
	 	\begin{enumerate}
	 		\item For any $r \ge 1$, the degree $rn$ term (with respect to the auxiliary variable $z$) of the plethystic exponential of a graded Hodge--Deligne polynomial $HD(X)(u, v) z^n$ of an $n$-dimensional variety $X$ can be decomposed as products of complete homogeneous polynomials with terms of the form $u^p v^q, u^q v^p, u^{n - p} v^{n - q}$ substituted in (only taking the first two in the case $p + q = n$) where the degrees depend on partitions of $n$ into $h^{p, q}$ parts. Note that each term of the plethystic exponential of $HD(X)(u, v) z^n$ has a degree in $z$ which is a multiple of $n$. \\
	 		
	 		\item The coloring polynomials (in the context of Theorem \ref{hdcolorpleth}) of Hodge--Deligne polynomials of the varieties $\mathbb{CP}^1, \mathbb{CP}^2$, and an elliptic curve $X$ generate the plethystic exponential of the Hodge--Deligne polynomial of a compact K\"ahler variety via nested determinants with ($\mathbb{Q}$-linear combinations of) coloring polynomials of complete graphs (or the graphs $G_k$ determining the chosen chromatic symmetric function basis). The determinants which we substitute into are determined by heights of rim hooks (p. 61 of \cite{ER}). \\
	 		
	 		\item Suppose we consider graded Hodge--Deligne polynomials of birationally equivalent varieties to be equivalent to each other. Let $h_n$ be the complete homogeneous symmetric polynomial of degree $n$ and $p_a$ be the power sum of degree $a$. If $n \ge 2$, the degree $n$ part (in the auxiliary variable) of a plethystic exponential is a $\mathbb{Z}$-linear combinations products of terms of the form $h_n \odot p_a$ with the number of terms in the product depending on a partition of $n$ into parts with the number of parts depending on the Hodge numbers as in Part 1. These generators depend on the difference of the Hodge--Deligne polynomials of the elliptic curve and $\mathbb{CP}^1$ and the Hodge--Delinge polynomial of $\mathbb{CP}^1$ itself. \\
	 	\end{enumerate}
	 \end{prop}
	 	
	 \begin{proof} ~\\
	 	\vspace{-5mm}
	 	\begin{enumerate}
	 		\item We can give some more detailed information on individual components or coefficients depending on the decomposition of the Hodge--Deligne polynomials that are chosen (e.g. generators as an algebra or an abelian group under addition). In the vector space case, we use the fact that the graded ring of formal Hodge--Deligne polynomials under the Hodge symmetry relation $h^{p, q} = h^{q, p}$ and the Serre duality relation $h^{p, q} = h^{n - p, n - q}$ consists of $\mathbb{Q}$-linear combinations (with nonnegative coefficients) of polynomials of the form $R_{p, q, n} := (u^p v^q + u^q v^p + u^{n - p} v^{n - q} + u^{n - q} v^{n - p} ) z^n$ for some $0 \le q \le p \le n$ such that $p + q \le n$. Let $S_{p, q}$ be the polynomial in $u$ and $v$ such that $S_{p, q} z^n = R_{p, q, n}$. 
	 		
	 		Recall that $PE[f] = \sum_{n \ge 0} (h_n \odot f)$. Since the polynomial $R_{p, q, n}$ is a sum of monomials,  a standard simplification (Theorem 7 on p. 171 of \cite{LR}) implies that \[ h_r \odot R_{p, q} = h_r(u^p v^q z^n, u^q v^p z^n, u^{n - p} v^{n - q} z^n, u^{n - q} v^{n - p} z^n) = h_r(u^p v^q, u^q v^p, u^{n - p} v^{n - q}, u^{n - q} v^{n - p}) z^{rn}. \] In the expression $h_r \odot R_{p, q, n} = h_r(u^p v^q, u^q v^p, u^{n - p} v^{n - q}, u^{n - q} v^{n - p}) z^{rn}$, we treat $h_r$ as the complete symmetric polynomial of degree $r$ in $4$ variables. The decomposition \[ h_r \odot (f + g) = \sum_{k = 0}^r (h_k \odot f) (h_{r - k} \odot g) \] implies that \[ h_r \odot (mF) = \sum_{i_1 + \ldots + i_m = r} (h_{i_1} \odot F) \cdots (h_{i_m} \odot F) \] and the generators $S_{p, q}$ are involved by taking $m = h^{p, q}$ as we vary over $0 \le q \le p \le n$ such that $p + q \le n$. This means substituting in $u^p v^q, u^q v^p, u^{n - p} v^{n - q},$ and $u^{n - q} v^{n - p}$ into $4$-variable homogeneous symmetric polynomials whose degrees depend on the partitions of $n$ into $h^{p, q}$ parts. \\
	 		
	 		Decomposing any formal Hodge--Deligne polynomial $F$ into a sum of the form $\sum_{\substack{0 \le q \le p \le n \\ p + q \le n}} h^{p, q} R_{p, q, n}$, the general $m$-term decomposition means that the degree $n$ parts can be written as \[ h_n \odot (F_1 + \ldots + F_m) = \sum_{i_1 + \ldots + i_m = n} (h_{i_1} \odot F_1) \cdots (h_{i_m} \odot F_m) \] taking $M = \{ (p, q) : 0 \le q \le p \le n, p + q \le n \}$ and $m = |M|$. Indexing each element of $M$ as $(p_r, q_r)$ for some $1 \le r \le m$, the coefficient of $z^{rn}$ in the plethystic exponential of the graded formal Hodge--Deligne polynomial is \[ h_n \odot F = \sum_{i_1 + \ldots + i_m = n} \prod_{r = 1}^m (h_{i_r} \odot S_{p_r, q_r}) =  \sum_{i_1 + \ldots + i_m = n} \prod_{r = 1}^m h_{i_r}(u^{p_{i_r} } v^{q_{i_r}}, u^{q_{i_r}} v^{p_{i_r}}, u^{n - p_{i_r}} v^{n - q_{i_r}}, u^{n - q_{i_r}} v^{n - p_{i_r}} ), \] where $m = \#  \{ (p, q) : 0 \le q \le p \le n, p + q \le n \}$. Using the analysis in Theorem \ref{hdcolorpleth}, each term of the product can be using colorings of acyclic orientations of $K_n$ by the ``variables $u^{p_{i_r}} v^{q_{i_r}}, u^{q_{i_r}} v^{p^{i_r}}, u^{n - p_{i_r}},$ and $v^{n - q_{i_r}}$. \\

	 		\item 
	 		
	 		As in Theorem 6 on p. 645 of \cite{KS} describing the structure of the formal Hodge--Deligne polynomials, we take $f$ to be a polynomial in variables $A$, $B$, and $C$ with $A = (1 + uv)z = HD(\mathbb{CP}^1)z$, $B = (u + v)z = (HD(X) - HD(\mathbb{CP}^1))z$, and $C = (uv)z^2 = (HD(\mathbb{CP}^1)^2 - HD(\mathbb{CP}^2))z^2$. The decomposition \[ h_n \odot (f + g) = \sum_{k = 0}^n (h_k \odot f) (h_{n - k} \odot g) \] reduces the computation to monomials $A^\alpha B^\beta C^\gamma$. Note that the value $m$ in the $m$-term generalization of this would be the coefficient of the monomial $A^\alpha B^\beta C^\gamma$ in the original Hodge--Deligne polynomial. \\
	 		
	 		These coefficients depend on the expression of $R_{p, q, n} = (u^p v^q + u^q v^p + u^{n - p} v^{n - q} + u^{n - q} v^{n - p} ) z^n$ as polynomials in $A, B, C$. In order to do this, we consider two recursions involving $A_m := (u^m + v^m) z^m$ and $T_r := (1 + u^r v^r) z^r$. This suffices since $R_{p, q, n} = (u^{p - q} + v^{p - q})(u^q v^q + u^{n - p} v^{n - p}) z^n = A_{p - q} C^{\min(q, n - p)} T_{|n - p - q|}$ for each $0 \le q \le p \le n$ such that $p + q \le n$. Since $A_1 = B, A_2 = B^2 - 2C$, and $A_s = BA_{s - 1} - C A_{s - 2}$, the single variable generating function associated to this sequence is $A(w) = \frac{Bw - 2}{Bw - Cw^2 - 1}$. Note that each $A_s$ is a polynomial in $B$ and $C$. Similarly, we have that $T_1 = A, T_2 = A^2 - 2C$, and $T_r = AT_{r - 1} - C T_{r - 2}$, which implies that $T(w) = \frac{Aw - 2}{Aw - Cw^2 - 1}$. The term $T_r$ is the coefficient of $w^r$ in the expansion of this rational function as a power series in $w$. In this case, the $T_r$ are each polynomials in $A$ and $C$. Putting these together, we obtain an expression for $R_{p, q, n}$ as a polynomial in $A, B, C$. To find the coefficient of a given monomial $A^\alpha B^\beta C^\gamma$ in a Hodge--Deligne polynmial, we take the sum over pairs $(p, q)$ with $0 \le q \le p \le n$ and $p + q \le n$ of the cofficient of $A^\alpha B^\beta C^\gamma$ in $R_{p, q, n}$ multiplied by $h^{p, q}$. \\  
	 		
	 		Next, we will focus on the individual monomials and use the multiplicative decomposition \[ h_n \odot (fg) = \sum_{|\lambda| = n} (s_\lambda \odot f) (s_\lambda \odot g) \] and Jacobi--Trudi identity $s_\lambda = \det(h_{\lambda_j + i - j})_{1 \le i, j \le n} = \sum_{\sigma \in S_n} \varepsilon(\sigma) h_{\lambda + \omega(\sigma)}$, where $\omega(\sigma)_j = \omega(j) - j$ and $\varepsilon(\sigma)$ is the sign of the permutation $\sigma$. Expanding with respect to the first row of this matrix, the $j^{\text{th}}$ term has an interpretation as the height of a special rim hook (proof of Theorem 1 on p. 64 -- 65 of \cite{ER}) This will not use the fact that we will eventually take $A, B,$ and $C$ to be sums of distinct monomials (which would reduce to the method used above) since we want to express things in terms of labeled coloring polynomials depending on $A, B,$ and $C$. Since $(fg) \odot h = (f \odot h) (g \odot h)$ and $(f + g) \odot h = (f \odot h) + (g \odot h)$, we have that $s_\lambda \odot F =  \det(h_{\lambda_j + i - j} \odot F)_{1 \le i, j \le n} = \sum_{\sigma \in S_n} \varepsilon(\sigma)( h_{\lambda + \omega(\sigma)} \odot F)$. For powers of individual variables such as $A^\alpha$, an induction argument implies that this is again generated by elements of the form $h_i \odot A$ (which can themselves be interpreted using graph colorings of acyclic directed (complete) graphs depending on the choice of chromatic symmetric polynomials by Theorem \ref{hdcolorpleth} and Corollary \ref{chromext}). Repeating the determinant construction shows that the conclusion follows from taking nested determinants (i.e. the determinant of a matrix whose entries are themselves determinants). \\

	 		The framework above is simplified we only consider polynomials up to birational equivalence of the varieties they arise from. Ae in the proof of Part 2, we will take $A_m := (u^m + v^m) z^m$ and $T_r := (1 + u^r v^r) z^r$. By Theorem 7 on p. 647 of \cite{KS}, we can set $C = 0$ in the original ring of formal Hodge--Deligne polynomials. Since $R_{p, q, n} = (u^{p - q} + v^{p - q})(u^q v^q + u^{n - p} v^{n - p}) z^n = A_{p - q} C^{\min(q, n - p)} T_{|n - p - q|}$ for each $0 \le q \le p \le n$ such that $p + q \le n$, this means that we only need to consider cases where $q = 0$ or $p = n$.  For $n \ge 2$, this essentially means that we are only allowed to have terms of the form $R_{n, 0, n} = (u^n + v^n + 2)z^n = (p_n + 2) z^n$, where $p_n = u^n + v^n$ is the $n^{\text{th}}$ power sum. The $n = 0$ case $R_{0, 0, 0} = 1$ and the $n = 1$ cases $R_{1, 0, 1} = 2(u + v)z = 2B$ and $R_{0, 0, 1} = 2(1 + uv) = 2A$ still remain the same. Since $h_n \odot (p_n + 2) = \sum_{a = 0}^n (h_a \odot p_n) (h_{n - a} \odot 2)$ and $h_r \odot 2 = r + 1$ (either directly or using the monomial simplification in Theorem 7 on p. 171 of \cite{LR}), we have that $h_n \odot (p_n + 2) = \sum_{a = 0}^n (n - a + 1) (h_a \odot p_n)$. Note that $p_n \sim B^n$ in the equivalence that we are considering. 
	 		
	 	\end{enumerate}
	 \end{proof}

	\section{Chromatic symmetric polynomials and symmetries of Hodge--Deligne polynomials} \label{chromsd}

	The chromatic symmetric polynomials used in Section \ref{plethgen} can also be used to analyze symmetries of the Hodge--Deligne polynomials themselves. This section focuses on their relation to the Serre duality-induced symmetry $h^{p, q} = h^{n - p, n - q}$ and equivariant Hodge--Deligne polynomials of configuration spaces of points. 
	
	\begin{prop} \label{chromhdsym} ~\\
		\vspace{-3mm}
		
		\begin{enumerate}
			\item The Serre duality-induced relation $h^{p, q} = h^{n - p, n - q}$ can be expressed in terms of algebraic relations of chromatic symmetric polynomials. 
			
			\item Given a quasiprojective variety $X$, let $F(X, n)$ be the ordered configuration space of $n$ distinct points on $X$. The equivariant Serre polynomial $\overline{e}_\sigma(F(X, n))$ (Definition \ref{equivser}) can be expressed as power series in chromatic symmetric polynomials. More specifically, it can be written as an infinite product where the $j^{\text{th}}$ term is of degree $j$ in these chromatic symmetric polynomials. 
		\end{enumerate}
	\end{prop}
	
	\begin{proof} ~\\
		
		\vspace{-5mm}
		
		\begin{enumerate}
			\item Since the chromatic symmetric polynomials are linear combinations of products of power sums, this reduces to an algebraic relation in terms of power sums. This applies the fact that degree $n$ homogeneous symmetric polynomials have chromatic symmetric polynomials as a $\mathbb{Q}$-basis (Theorem \ref{chromgen}) to a result of Florentino--Nozad--Zamora (Theorem 1.1 on p. 2 of \cite{FNZ}) making the connection between the generating functions of this character variety and that of one of its irreducible components. \\
						
			We can write the Serre duality-induced symmetry condition $h^{p, q} = h^{n - p, n - q}$ in terms of products of power sums (which can then be converted into a relation involving chromatic symmetric polynomials). This involves a relation with products of power sums. Since any symmetric polynomial can be written as a polynomial in the elementary symmetric polynomials $e_k$, the Serre duality identity $(uv)^n \overline{e}(X)(u^{-1}, v^{-1}) = \overline{e}(X)(u, v)$ can be rewritten as $(uv)^n P(\widetilde{e}_1, \ldots, \widetilde{e}_n) = P(e_1, \ldots, e_n)$ with $\widetilde{e}_k(u, v) := e_k(u^{-1}, v^{-1}) = \frac{e_{n - k(u, v)}}{e_n(u, v)}$. By Proposition 7.7.5 on p. 300 of \cite{Sta2}, we have that $e_r = \sum_{\lambda \dashv r} e_\lambda z_\lambda^{-1} p_\lambda$ with $e_\lambda = (-1)^{m_2 + m_4 + \ldots}$ for a partition $\lambda = (1^{m_1} 2^{m_2} \cdots)$ of $r$ (p. 299 of \cite{Sta2}). In the context of permutation representations, this depends on whether a permutation is even or odd. \\
			
			Let $a_m = (uv)^m$. Then, we have that $a_m = e_{2m} - p_{2m} - a_1 p_{2m - 2} - \ldots - a_{m - 1} p_2$ since $e_r = u^r + u^{r - 1} v + \ldots + u v^{r - 1} + v^r$. Also, note that $u^r v^s + u^s v^r = p_r p_s - p_{r + s}$. Since $e_r = \sum_{\lambda \dashv r} z_\lambda^{-1} p_\lambda$ and $(uv)^m = e_m - p_m - \sum_{r = 1}^{m - 1} (u^r v^{m - r} + u^{m -  r} v^r) = e_m - p_m - \sum_{r = 1}^{m - 1} (p_r p_{m - r} - p_m)$, we can express $(uv)^m$ as a $\mathbb{Q}$-linear combination of power sums $p_\lambda$ with $\lambda \dashv 2m$. This means that $(uv)^n P(\widetilde{e}_1, \ldots, \widetilde{e}_n) = (uv)^n P \left( \frac{e_{n - 1}}{e_n}, \frac{e_{n - 2}}{e_n}, \ldots, \frac{e_1}{e_n}, \frac{1}{e_n} \right) = P(e_1, \ldots, e_n)$ (a polynomial relation in the $e_r$) can be rewritten as a polynomial relation in products of power sums. Clearing denominators and collecting power sums $p_\lambda$ with the same $|\lambda| = \lambda_1 + \ldots + \lambda_r$, this can be rewritten as a relation involving chromatic symmetric polynomials $X_G$. \\
			
			Finally, we can compare the chromatic polynomial decomposition of the original Hodge--Deligne polynomial with that of the plethystic exponential in the previous section. The work above makes some comments on the orignal one since we expressed $u^r v^s + u^s v^r$ as a $\mathbb{Q}$-linear combination of products of power sums $p_\lambda$. This can be compared with the plethysm decomposition in the previous section.

			\item This is an application of Theorem \ref{chromgen} to the following result of Getzler \cite{G}:
			
				\begin{thm} (Getzler, p. 2 of \cite{G}) \label{equivserconf} \\
					Let $X$ be a quasiprojective vairety, and let $\alpha_j(X) = \sum_{d | j} \mu \left( \frac{j}{d} \right) e(X; u^d, v^d)$. \\
					
					If $\sigma \in S_n$ is a permutation with $n_j$ cycles of length $j$, then \[ e_\sigma(F(X, n)) = \prod_{j = 1}^\infty \alpha_j(X) (\alpha_j(X) - j) \cdots (\alpha_j(X) - (n_j - 1)j), \]
					
					where $F(X, n)$ is the configuration space of $n$ ordered points on $X$. 
				\end{thm}
			
			The result follows from writing $\alpha_j(X) = \sum_{d | j} \mu \left( \frac{j}{d} \right) e(X; u^d, v^d)$ as a polynomial in chromatic symmetric polynomials. In order to obtain an actual symmetric polynomial, we will actually look at the same polynomial $-u$ and $-v$ substituted in for $u$ and $v$, which we denote by $\overline{\alpha}_j(X)(u, v) := \alpha_j(X)(-u, -v)$. Similarly, we can set $\overline{e}(X)(u, v) = e(X)(-u, -v)$. While $\overline{\alpha}_j(X)$ is a symmetric polynomial, it is not homogeneous. We can split $\overline{\alpha}_j(X)$ into a sum where each term is a \emph{homogeneous} symmetric polynomial by partitioning by degrees. Then, Theorem \ref{chromgen} implies  $\overline{\alpha}_j(X)$ can be written as a sum of chromatic symmetric polynomials associated to a set of connected graphs $\{ G_k \}$ with $G_k$ a connected graph with $k$ vertices. 
		\end{enumerate}
	\end{proof}
	
	\begin{rem}
		Some amusing connections with enumerative questions about graphs are that the relation  $PExp \left( \sum_{p, q \ge 0} a_{p, q} u^p v^q y \right)  = \prod_{p, q \ge 0} (1 - u^p v^q)^{-a_{p, q}}$ is similar to the one between the generating function for the number of graphs with a given number of points and edges and the number of \emph{connected} graphs with these properties (equation between (28) and (29) on p. 457 of \cite{Ha}) and that the Serre duality-induced relation $h^{p, q} = h^{n - p, n - q}$ is somewhat similar to the relation $g_{p, k} = g_{p, p(p - 1)/2 - k}$ satisfied by the connected graph generating functions (p. 451 of \cite{Ha}).  
		
	\end{rem}

Department of Mathematics, University of Chicago \\
5734 S. University Ave, Office: E 128 \\ Chicago, IL 60637 \\
\textcolor{white}{text} \\
Email address: \href{mailto:shpg@uchicago.edu}{shpg@uchicago.edu} 
\end{document}